\documentclass{amsart}

\usepackage[utf8]{inputenc}
\usepackage[T1]{fontenc}
\usepackage{verbatim}
\usepackage{graphicx}
\usepackage{graphicx,caption2,psfrag,float,color}
\usepackage{amssymb}
\usepackage{amscd}
\usepackage{amsmath}

\usepackage[T1]{fontenc}

\newtheorem{theorem}{Theorem}[section]

\newtheorem{lemma}[theorem]{Lemma}

\numberwithin{equation}{subsection}
\newtheorem{definition}[theorem]{Definition}

\title{The ternary Goldbach problem with a prime and two isolated primes}
\author{Helmut Maier and Michael Th. Rassias}
\date{\today}
\address{Department of Mathematics, University of Ulm, Helmholtzstrasse 18, 89081 Ulm, Germany.}
\email{helmut.maier@uni-ulm.de}
\address{Institute of Mathematics, University of Z\"urich, Winterthurerstrasse 190, CH-8057 Zürich, Switzerland  \& Institute for Advanced Study, Einstein Drive, Princeton, New Jersey 08540 USA}
\email{michail.rassias@math.uzh.ch, michailrassias@math.princeton.edu}
\thanks{}

\begin{document}

 \maketitle
 
\begin{abstract} In the present paper we prove that under the assumption of the GRH (Generalized Riemann Hypothesis) each sufficiently large odd integer can be
expressed as the sum of a prime and two isolated primes.\\

%\noindent\textbf{2010 Mathematics Subject Classification:}  11P32, 26D20%
%\newline\noindent\textbf{Key words:} Gaps of primes, sieve estimate, sieve weight, covering theorem, arithmetic progression.

\end{abstract}

\section{Introduction}
In 1937 Vinogradov \cite{vino} proved that each sufficiently large odd integer is the 
sum of three primes. For a step by step exposition of Vinogradov's theorem see 
\cite{rassias}. Recently Helfgott \cite{helfgott} (see also \cite{helfgottplatt}) proved
the result that all odd integers greater than five have this property. Wirsing \cite{wirsing}, motivated by earlier work of Erd\H{o}s and Nathanson \cite{erdos}
 on sums of squares, considered the question of whether one could find thin subsets
 $S$ of primes, which are still sufficient to obtain all sufficiently large odd integers as
 sums of three of them and obtained the answer that there exist such sets $S$ with
 $$\sum_{p\leq x, p\in S}1\ll(x\log x)^{1/3}\:.$$
 Wirsing's result used probabilistic methods and did not lead to a subset of the 
 primes, which is constructive. Wolke in an Oberwolfach conference in 1986
 suggested to find more familiar thin sets of primes. This was achieved by Balog
and Friedlander \cite{balog} who merged the result of Vinogradov with that of 
Piatetski-Shapiro \cite{Piatetski}, who considered sets
$$P_{C}=\{ p\::\:p\: \text{prime},\ p=[n^C]\}$$
with a fixed constant $C>1$.\\
Piatetski-Shapiro \cite{Piatetski} proved that 
$$\sum_{\substack{n\leq x \\ [n^C]\: prime}}1=(1+o(1))\:\frac{x}{C\:\log x}$$
holds in the range $1<C<12/11$, which later was improved by many authors.\\
Balog and Friedlander proved (\cite{balog}, Corollary 1):

\textit{For any fixed $C$ with $1< C<21/20$ every sufficiently large odd integer can be written as the sum of three primes from $P_C$.}

\noindent They also prove more general results in which the three primes are taken from sets $P_C$ with possibly different values of $C$.\\
Here we consider another special set of primes: \textbf{isolated primes}.
\begin{definition}
Let $g\::\:\mathbb{N}\rightarrow [1,\infty)$ be a monotonically increasing function with $g(n)\rightarrow \infty$ for $n\rightarrow \infty$. We say that a prime $p$ is \textbf{$g$-isolated}, if $m$ is composite for all positive integers $m$ with 
$$0<|p-m|\leq (\log p)g(p)\:.$$
\end{definition}
This concept of isolation is closely linked to the question of large gaps between primes. Let $p_n$ be the $n$th prime. For which functions $g$ do we have:
\[
p_{n+1}-p_n\geq (\log p_n)g(p_n)\tag{1.1}
\]
infinitely often?\\
After Westzynthius \cite{Westzynthius} had shown that (1.1) holds for $g$ being an arbitrarily large positive constant, further progress was achieved by variations of the Erd\H{o}s-Rankin method (\cite{erdos2}, \cite{Rankin}). Rankin \cite{Rankin}
could show that (1.1) holds with
\[
g(p_n)=C\:\frac{\log_2 p_n\: \log_4 p_n}{\log_3^2 p_n}\:,\tag{1.2}
\]
($C$ a positive constant, $\log_1 n=\log n$, $\log_k n=\log(\log_{k-1} n)$).\\
For a long period of time the improvements of the result (1.2) only involved the constant $C$ (\cite{maier}, \cite{Pintz}, \cite{Rankin2}, \cite{Schonhage}). A famous
prize problem of Paul Erd\H{o}s was the improvement of the order of magnitude of the function $g$. This was achieved only recently (\cite{ford}, \cite{ford2}, \cite{Maynard}). The latest
result (\cite{ford2}) is:
$$g(p_n)=C\:\frac{\log_2 p_n\: \log_4 p_n}{\log_3 p_n}\:. $$
We shall use this function $g$ in our definition of isolated primes, which in the sequel we shall simply call isolated.
\begin{definition}
Let $C$ be a fixed positive constant, which we assume to be sufficiently small. We say, a prime number $p$ is \textbf{isolated} if $\log_4 p\geq 1$ and $m$ is composite
for all integers $m$ with
$$0<|p-m|\leq C\: \log p\:\frac{\log_2 p\:\log_4 p}{\log_3 p}\:.$$
\end{definition}
Our result is the following:
\begin{theorem}\label{main}
Assume the Generalized Riemann Hypothesis. Then each sufficiently large odd integer is a sum of a prime and two isolated primes.
\end{theorem}
\textit{Remark.} There are several challenges. It is very likely true that each sufficiently large odd integer is the sum of three isolated primes. It would be worthwhile to find a proof of this fact, possibly with a function $g$ of smaller order of 
magnitude. It also remains the problem to remove the assumption of the Generalized 
Riemann Hypothesis.
\section{Construction of the isolated residue-class}
Let $p_1, p_2$ be two isolated primes in the representation $N=p_1+p_2+p_3$.
In this section we shall construct a modulus $P^*$ being a product of many small
prime numbers. We also construct an ``isolated'' residue-class $u_0(\bmod\:P^*)$,
such that $(u_0, P^*)=(N-2u_0, P^*)=1$ and $(m, P^*)>1$
for $m\neq u_0$ if $|u_0-m|$ is small.\\
The proof of our Theorem \ref{main} will then be concluded in the next section by the
circle method, choosing $p_1$ and $p_2$ from the residue class $u_0(\bmod\:P^*)$
and $p_3$ from the residue class $(N-2u_0)\bmod\:P^*$. From 
$$(p_1+n,P^*)=(p_2+n, P^*)=(u_0+n, P^*)>1\:,$$
it then follows that $p_1$ and $p_2$ are isolated in the sense as defined in the previous section.
More specifically we shall prove in this section:
\begin{theorem}\label{thm2}
Let $C>0$ and $D>0$ be fixed, $C$ sufficiently small and $D$ sufficiently large.
Let $N=N(C,D)$ be a sufficiently large odd integer. Then there is a positive integer $P^*$ with 
$$(P^*)^{100}\leq N\leq (P^*)^D$$
and an integer $u_0$ with 
$$(P^*)^C\leq u_0\leq P^*\:,$$
such that 
$$(u_0, P^*)=(N-2u_0, P^*)=1\ \ \text{and}\ \ (m,P^*)>1$$
for
$$0<|u_0-m|\leq C(\log N)\: \frac{\log_2 N\:\log_4 N}{\log_3 N}\:.$$
\end{theorem}
Our construction is mainly based on the ideas of the paper \cite{ford2}. Also our
definitions are mainly taken from \cite{ford2}.\\
Let $c_1, c_2$ be fixed positive constants to be specified later. Also the constants $c_3, c_4, \ldots$ will be positive constants. They will depend only on $c_1, c_2$. 
We set
\[
x:=c_1\: \log N\tag{2.1}
\]
\[
y:=c_2x\: \frac{\log x\: \log_3 x}{\log_2 x} \tag{2.2}
\]
\[
z:=x^{\log_3 x/(4\log_2 x)} \tag{2.3}
\]
\begin{definition}\label{def21}
We introduce the three disjoint sets of primes
\[
S:=\{s\ \text{prime}\::\: \log^{20} x< s\leq z\} \tag{2.4}
\]
\[
\mathcal{P}:=\{p\ \text{prime}\::\: x/2< p\leq x\} \tag{2.5}
\]
\[
Q:=\{q\ \text{prime}\::\: x< q\leq y\} \tag{2.6}
\]
For vectors $\vec{a}=(a_s \bmod\: s)_{s\in S}$, $\vec{b}=(b_p \bmod\: p)_{p\in \mathcal{P}}$,
we define the sifted sets
$$S(\vec{a}):=\{n\in\mathbb{Z}\::\: n\not\equiv a_s(\bmod\: s)\ \text{for all}\ s\in S\}$$
and likewise 
 $$S(\vec{b}):=\{n\in\mathbb{Z}\::\: n\not\equiv b_p(\bmod\: p)\ \text{for all}\ p\in \mathcal{P}\}$$
\end{definition}
\begin{definition}\label{def22}
For each prime $p\leq x$ we define:
$$d_p:=\left\{ 
  \begin{array}{l l}
    a_s\:, & \quad \text{if $p=s\in S$}\vspace{2mm}\\ 
     b_p\:, & \quad \text{if $p\in \mathcal{P}$}\vspace{2mm}\\ 
      0\:, & \quad \text{for all other $p$}
  \end{array} \right.
$$
\end{definition}
\begin{lemma}\label{lem21}
Let $n\in(x,y]$ satisfy $n\not\equiv d_p(\bmod\: p)$ for all $p\leq x$. Then 
\mbox{$n\in Q\cap S(\vec{a})\cap S(\vec{b})$} or $n\in R$, where
$$R:=\{n\in(x, y]\::\: p\mid n \Rightarrow p\leq z\}$$
is the set of $z$-smooth integers in $(x, y]$.
\end{lemma}
\begin{proof}
Assume $n\not\equiv d_p(\bmod\: p)$ for all $p\leq x$ and $n\not\in Q$. Assume there are two primes $t_1, t_2$ with $t_1<t_2, t_2>z$ and $t_1\mid n$, $t_2\mid n$.\\
Then by (2.4), (2.5) we have: 
$$n\geq t_1t_2\geq \frac{x}{2}(\log x)^{20}>y\:,$$
a contradiction. Thus, since $n\not\in Q$, we have $n\in R$.
\end{proof}
\begin{lemma}\label{lem22}
It holds
$$\# R=O\left(\frac{x}{(\log x)^3} \right)\:.$$
\end{lemma}
\begin{proof}
To estimate $\# R$, let 
$$u:=\frac{\log y}{\log z}\:.$$
So from (2.2), (2.3) one has 
$$u=4\:\frac{\log_2 x}{\log_3 x}(1+o(1)).$$
By standard estimates for smooth numbers (e.g. de Bruijn's theorem \cite{Bruijn} and (2.2))
we have:
$$\#R\ll ye^{-u\log u+O(u\log\log(u+2))}=\frac{y}{\log^{4+o(1)} x}=O\left(\frac{x}{\log^3 x}\right)\:.$$
\end{proof}
\begin{lemma}\label{lem23}
Let $N$ be sufficiently large and suppose that $x, y, z$ are given by (2.1), (2.2) and
(2.3). Then there are vectors $\vec{a}=(a_s \bmod\: s)_{s\in S}$ and 
$\vec{b}=(b_p \bmod\: p)_{p\in \mathcal{P}}$, such that 
\[
\#(Q\cap S(\vec{a})\cap S(\vec{b})\ll\frac{x}{\log x}\:.\tag{2.5}
\]
\end{lemma}
\begin{proof}
This follows immediately from Theorem 2 of \cite{ford2}.
\end{proof}
We now fix $\vec{a}$ and $\vec{b}$ satisfying (2.5). We need some results from standard sieves (Brun's or Selberg's sieve). We borrow the following notations and 
results from \cite{Halberstam}. An exception is the set of primes, denoted by $\mathcal{P}$
in \cite{Halberstam}, which we denote by $\tilde{\mathcal{P}}$.
\begin{definition}\label{def23}
Let $\mathcal{A}$ be a finite set of integers and let $\tilde{\mathcal{P}}$ be a set of primes. For a positive squarefree integer $d$ composed only of primes of
$\tilde{\mathcal{P}}$ let
$$\mathcal{A}_d:=\{ n\in\mathcal{A}\::\: n\equiv 0 (\bmod\: d) \}\:.$$ 
\end{definition} 
Let $z$ be a positive real number and let $P(z)$ be a product of the primes in 
$\tilde{\mathcal{P}}$ that are smaller than $z$. Then we set
\[
S(\mathcal{A}; \tilde{\mathcal{P}}, z):=\left|\left\{a\::\: a\in\mathcal{A},\: (a, P(z))=1  \right\} \right|\:.\tag{2.6}
\]
Let $\omega$ be a multiplicative function defined for squarefree numbers with $\omega(p)=0$ for 
$p\not\in \tilde{\mathcal{P}}$. With $X$ an appropriate constant we set:
\[
R_d:= |\mathcal{A}_d|-\frac{\omega(d)}{d}\: X\:.\tag{2.7}
\]
We also define
\[
W(z):=\prod_{p< z}\left(1-\frac{\omega(p)}{p}\right)\:.\tag{2.8}
\]
We introduce the conditions:
\[
0\leq \frac{\omega(p)}{p}\leq 1-\frac{1}{A_1} \tag{$\Omega_1$}
\]
for some suitable constant $A_1\geq 1$.
\[
\sum_{w\leq p< z}\frac{\omega(p)\log p}{p}\leq \kappa\log\frac{z}{w}+A_2\:, \tag{$\Omega_2(\kappa)$}
\]
if $2\leq w\leq z$, where $\kappa\:(>0)$ and $A_2\:(\geq 1)$ are independent of $z$ and
$w$.
\[
|R_d|\leq L\left( \frac{X\log X}{d}+1\right)A_0'^{\nu(d)}\:,\ \ \text{for $\mu(d)\neq 0$, $(d, \overline{\mathcal{P}})=1$,}\tag{$R_0$}
\]
where $L$ and $A_0'$ are suitable constants, 
$$\overline{\mathcal{P}}=\{p\in P\::\: p\not\in \mathcal{P}\}\:.$$
\text{Let $\alpha$ be a constant with $0<\alpha\leq 1$}.\\ 
Assume that for any constant $u\geq 1$ there exists a constant $C_0>0$ such that
\[
\sum_{\substack{d<X^\alpha\log^{-C_0}X\\ (d, \overline{\mathcal{P}})=1}} \mu^2(d)|R_d|=O_u\left(\frac{X}{\log^{\kappa+u}X}\right)\:.\tag{$R_1(\kappa, \alpha)$}
\]
\begin{lemma}\label{lem24}
Let $\mathcal{A}$ satisfy the conditions $(\Omega_1)$,  $(\Omega_2(\kappa))$, $(R_0)$,
$(R_1(\kappa,\alpha))$. Let  $X\geq z$ and write 
$$u=\frac{\log X}{\log z}.$$
Then 
\begin{align*}
S(\mathcal{A}; \tilde{\mathcal{P}}, z)=&XW(z)\{1+O(\exp\{-\alpha u(\log u-\log\log 3u-\log(x/\alpha)-2) \})\tag{2.9}\\
&\ \ \ \ +O_u(L\log^{-u}X)\:,
\end{align*}
where the $O$-constants may depend on $u$ as well as on the constants $A_1'$,
$A_1$, $A_2$, $\kappa$ and $\alpha$.
\end{lemma}
\begin{proof}
This is Theorem 2.5$'$ of \cite{Halberstam}.
\end{proof}
\begin{definition}\label{def24}
Let 
$$Q^*=\left\{q\in Q\::\: \frac{y}{3}<q\leq \frac{2y}{3},\ p\mid N-2q \Rightarrow p>\log^{20}x\right\}\:.$$
\end{definition}
\begin{lemma}\label{lem25}
$$\#Q^*\geq c_3\:\frac{y}{\log x\: \log_2 x}\:.$$
\end{lemma}
\begin{proof}
We apply Lemma \ref{lem24} with
$$\mathcal{A}=\left\{ N-2q\::\: \frac{y}{3}<q\leq \frac{2y}{3} \right\},\ z=(\log x)^{20}\:,$$
$$X=li\:\frac{2y}{3}-li\:\frac{y}{3},\ 
\omega(p)=\left\{ 
  \begin{array}{l l}
    \frac{p}{p-1}\:, & \quad \text{for $p>2$}\vspace{2mm}\\ 
      0\:, & \quad \text{for $p=2$}
  \end{array} \right.
$$
Then the conditions $(\Omega_1)$, $(\Omega_2(\kappa))$ are satisfied with $A_1=2$,
$\kappa=1$. The existence of $A_2$ follows from Mertens' result
$$\sum_{p\leq x}\frac{\log p}{p}=\log x+O(1)\:.$$
By the Generalized Riemann Hypotehsis we have 
$$|R_d|\leq X^{1/2}\log^3 X\:,\ \text{for $d\leq X^{1/2}$}.$$
We have
$$W(z)=\prod_{p<(\log x)^{20}}\left(1-\frac{\omega(p)}{p}\right)=O\left(\frac{1}{\log\log x}\right)\:.$$
Lemma \ref{lem25} now follows from (2.9).
\end{proof}
\begin{definition}\label{def25}
For $w\geq 1$, $r, b \in\mathbb{N}$, let
$$\pi(w, r, b)=\#\{ p\leq w\::\: p\equiv b(\bmod\: r) \}\:.$$
\end{definition}
\begin{lemma}\label{lem26}
For $1\leq r\leq b$, $(b,r)=1$, we have:
$$\pi(w, r, b)=O\left(\frac{w}{\phi(r)\log(w/r)} \right)\:.$$
\end{lemma}
\begin{proof}
This is the Brun-Titchmarsh Inequality (see \cite{Halberstam}, Theorem 3.8).
\end{proof}
\begin{definition}\label{def26}
Let $q\in Q^*$. Then we define
$$\mathcal{N}_1(q):=\{s\in S\::\: q\equiv a_s(\bmod\: s)\}$$
$$\mathcal{N}_2(q):=\{p\in \mathcal{P}\::\: q\equiv b_p(\bmod\: p)\}$$
$$\mathcal{N}_3(q):=\{p\leq x \::\: N-2q\equiv 0(\bmod\: p)\}\:.$$
\end{definition}
\begin{lemma}\label{lem27}
Let $x\rightarrow \infty$. Then we have for all but $o(|Q^*|)$ primes $q\in Q^*$:
\[
\#\mathcal{N}_1(q)=O(\log x\:\log_2 x\:\log_3 x) \tag{2.10}
\]
\[
\#\mathcal{N}_2(q)=O((\log_2 x)^2) \tag{2.11}
\]
\[
\#\mathcal{N}_3(q)=O((\log\log x)^3) \tag{2.12}
\]
\end{lemma}
\begin{proof}
We have by Lemma \ref{lem26}:
$$\sum_{q\in Q^*}\# \mathcal{N}_1(q)=\sum_{s\in S}\ \sum_{\substack{q\equiv a_s(\bmod\: s)\\ q\in Q^*}}1=O\left(\frac{y}{\log x}\: \sum_{s\in S}\frac{1}{s} \right)=
O\left(\frac{y}{\log x}\:\frac{\log x\:\log_3 x}{\log_2 x}\right).$$
The relation (2.10) follows from Lemma \ref{lem25}.\\
We have 
$$\sum_{q\in Q^*}\#\mathcal{N}_2(q)=\sum_{p\in \mathcal{P}}\ \sum_{\substack{q\equiv b_p(\bmod\: p)\\ q\in Q^*}}1=O\left(y\: \sum_{p\in \mathcal{P}}\frac{1}{p}\right)=O\left(\frac{y}{\log x}\right)\:.$$
The relation (2.11) now follows from Lemma \ref{lem25}. We have
$$\sum_{q\in Q^*}\#\mathcal{N}_3(q)=\sum_{(\log x)^{20}<p\leq x}\sum_{\substack{q\in Q^* \\ N-2q\equiv 0(\bmod p)}}1=O\left(y\:\sum_{(\log x)^{20}<p\leq x}\frac{1}{p}\right)=O(y\log\log x)\:.$$
The relation (2.12) now follows from Lemma \ref{lem25}.
\end{proof}
\begin{lemma}\label{lem28}
There is a prime $q_0\in \left[\frac{y}{3},\frac{2y}{3}\right]$ and a subset
$$\mathcal{P}^*\subset\{p\leq x\::\: p\ \text{prime}\}\:,$$
such that $q_0\not\equiv d_p(\bmod\: p)$ for $p\in \mathcal{P}^*$, $p\nmid N-2q_0$
for $p\in \mathcal{P}^*$.
$$\#\{n\in(x,y]\::\: n\equiv d_p(\bmod\: p)\ \text{for $p\in\mathcal{P}^*$}\}=O\left(\frac{x}{\log x}\right)\:.$$
\end{lemma}
\begin{proof}
Since the sets $\{n\equiv d_p(\bmod\: p)\}$ are arithmetic progressions with difference $p$, we have by Lemma \ref{lem27}:
\begin{align*}
\#\{n\in(x,y]\::\: n\equiv a_s(\bmod\: s)\ \text{for some $s\in \mathcal{N}_1(q_0)$}\}
&=O(x\log x \log_2 x\log_3 x(\log x)^{-20})\\
&=O(x(\log x)^{-19})\:,\tag{2.13}
\end{align*}
\begin{align*}
\#\{n\in(x,y]\::\: n\equiv b_p(\bmod\: p)\ \text{for some $s\in \mathcal{N}_2(q_0)$}\}
&=O(x\log x (\log_2 x)^2x^{-1})\\
&=O(x(\log x)^{-19})\tag{2.14}
\end{align*}
and 
\[
\#\{n\in(x,y]\::\: n\equiv d_p(\bmod\: p)\ \text{for some $p\in \mathcal{N}_3(q_0)$}\}
=O(x(\log x)^{-19})\:.\tag{2.15}
\]
We now define 
\[
\mathcal{P}^*:=\{p\ \text{prime},\ p\leq z\}\setminus \{\mathcal{N}_1(q_0)\cup \mathcal{N}_2(q_0) \cup \mathcal{N}_3(q_0)  \}\:.\tag{2.16}
\]
The claim of Lemma \ref{lem28} follows from Lemmas \ref{lem21}, \ref{lem22},
\ref{lem23} and the estimates (2.13), (2.14) and (2.15). 
\end{proof}
We now conclude the proof of Theorem \ref{thm2}.\\
Let $n_1, \ldots, n_k\in (x, y]\setminus \{q_0\}$ with $n_j\not\equiv d_p(\bmod\: p)$
for all $p\in \mathcal{P}^*.$ By the prime number theorem there is a $C_0>1$, such
that 
$$\pi(C_0x)-\pi(x)\geq 2k\:.$$ 
Since $p\mid (n_j-q_0)$ for at most one prime
$p\in (x, C_0x]$ we may choose primes $\tilde{p}_1, \ldots, \tilde{p}_k\in(x, C_0x]$,
such that $n_j\not\equiv q_0 (\bmod\: p)$ for $1\leq j\leq k$.
\begin{definition}\label{def 27}
We define 
$$d\tilde{p}_j=n_j,\ \text{for $1\leq j\leq k$.}$$
\end{definition}
\noindent We now set 
\[
P^*=\prod_{p\in\mathcal{P}^*}p\prod_{j=1}^k \tilde{p}_j \tag{2.17}
\]
and determine $v_0$ by the conditions
\begin{align*}
&1\leq v_0< P^*\\
&v_0\equiv -d_p(\bmod\: p)\ \text{for all $p\in P^*$}\:.\tag{2.18}
\end{align*}
By the Chinese Remainder Theorem $v_0$ is uniquely determined. We now 
define the isolated residue class $(\bmod\: P^*)$ by 
\[
u_0=v_0+q_0\:.\tag{2.19}
\]
From (2.18) we see that $p\mid v_0+d_p$ for all $p\in \mathcal{P}^*\setminus \{q_0\}$. Thus, we have 
$(n, P^*)>1$ for all $n\in (x, y]\setminus \{u_0\}$. Thus $|u_0-m|$ is composite
for all $m$ with $0<|m|\leq y/3$.  By Definition \ref{def26} for $\mathcal{N}_3(q)$
and (2.16) we have:
$$(u_0, P^*)=(N-2u_0, P^*)=1\:.$$
We have 
\[
P^*\leq P(C_0x)=\exp(C_0x(1+o(1)))\tag{2.20}
\]
by the prime number theorem.\\
By the Definition \ref{def22} we have $d_p=0$ for $z<p\leq x/2$. Therefore
$$u_0\geq \prod_{\substack{z< p\leq x/2\\ p\not\in \mathcal{N}_1(q_0)\cup  \mathcal{N}_2(q_0) \cup  \mathcal{N}_3(q_0)}} p\geq (P^*)^c$$
if $c$ is chosen small enough.\\
The bound $N\geq (P^*)^{100}$ follows, if $c_1$ in (2.1) is chosen sufficiently small.\\
The upper bound $N\leq (P^*)^D$ follows from Lemma \ref{lem27}. This concludes the proof of 
Theorem \ref{thm2}.
\section{The Circle Method}
Let $N$ and $P^*$ satisfy the conditions of Theorem \ref{thm2} and let $u_0$ be
the isolated residue-class $\bmod\: P^*$. It remains to be shown that there are primes $p_1, p_2, p_3$ with $p_1\equiv p_2\equiv u_0(\bmod\: P^*)$, 
$p_3\equiv N-2u_0(\bmod\: P^*)$ with 
$$p_1+p_2+p_3=N\:.$$
This will be achieved by the circle method.\\
We closely follow \cite{liu}. The results and definitions are borrowed from there with slight modifications. They are formulated for general arithmetic progressions.
\begin{definition}\label{def31}
Let $R\leq N^{1/20}$, $P=R^3L^{3C}$, $Q=NR^{-3}L^{-4C}$, $L=\log N$, with the
constant $C$ to be specified later. The major arc of the circle method is defined as
\[
E_1(R):=\bigcup_{q\leq P}\bigcup_{\substack{a=1\\ (a,q)=1}}^q\left[\frac{a}{q}-\frac{1}{qQ},\: \frac{a}{q}+\frac{1}{qQ} \right]\tag{3.2}
\]
\[
E_2(R):=\left[\frac{1}{Q},1+\frac{1}{Q} \right]-E_1(R)\:.\tag{3.3}
\]
Since $2P<Q$, no two major arcs intersect. Write $\alpha\in [0,1]$ in the form 
\[
\alpha=\frac{a}{q}+\lambda,\ 1\leq a\leq q,\ (a,q)=1\:.\tag{3.4}
\]
It follows from Dirichlet's lemma on rational approximation that
$$E_2(R)=\left\{ \alpha\::\: P<q<Q,\ |\lambda|\leq \frac{1}{qQ} \right\}\:.$$
Let $\Lambda(n)$ be the von Mangoldt function, $e(\alpha)=e^{2\pi i \alpha}$ and
\[
S(\alpha, r, b)=\sum_{\substack{n\leq N \\ n\equiv b(\bmod\: r)}}\Lambda(n)e(n\alpha)\:.\tag{3.5}
\]
\end{definition}
\noindent By orthogonality we then have 
\[
\sum_{\substack{(n_1, n_2, n_3)\in\mathbb{N}^3,\\ n_i\equiv b_i(\bmod r)\\ n_1+n_2+n_3=N}} \Lambda(n_1)\Lambda(n_2)\Lambda(n_3)=\int_0^1 S(\alpha, r, b_1)S(\alpha, r, b_2)S(\alpha, r, b_3)e(-N\alpha)d\alpha\:.\tag{3.6}
\]
\begin{lemma}\label{lem31}
Let $A>0$ be arbitrary and $\alpha\in E_2(R)$. If $C$ is sufficiently large, then 
\[
S(\alpha, r, b) \ll\frac{N}{r\log^A N}\tag{3.7}
\]
uniformly for $\frac{1}{2}R\leq r\leq 2R$.
\end{lemma}
\begin{proof}
This follows from the result of Balog and Perelli \cite{balog2} which we present below:\\
For $M\leq N$ and $h=(r, q)$ it holds
$$\sum_{\substack{n\leq M \\ n\equiv b(\bmod r)}}\Lambda(n)e\left(\frac{a}{q}\: n\right)\ll L^3\left( \frac{hN}{rq^{1/2}}+\frac{q^{1/2}N^{1/2}}{h^{1/2}}+\frac{N^{4/5}}{r^{2/5}}\right)\:.$$
\end{proof}
\begin{definition}\label{def32}(Definitions from \cite{liu})\\
Let $d, f, g, k, m$ be fixed positive integers and $\chi_g$ a Dirichlet character $\bmod g$. Let 
\[
G(d, f, m, \chi_g, k)=\sum_{\substack{n=1 \\ (n,k)=1\\ n\equiv f(\bmod\: d)}}^k\chi_g(n)e\left( \frac{mn}{k}\right)\tag{3.8}
\]
\textit{Remark.} This is a generalization of the Gaussian sum
$$G(m,\chi)=\sum_{n=1}^k\chi(n)e\left( \frac{mn}{k}\right)\:.$$
In the special case $g=k$, let
\[
G(d, f, m, \chi)=G(d, f, m, \chi_k, k)\:.\tag{3.9}
\]
For positive integers $r$ and $q$ let 
\[
h=(r,q)\:.\tag{3.10}
\]
Then $r, q$ and $h$ can be written as 
\begin{align*}
&r=\tilde{p}_1^{\alpha_1}\cdots\: \tilde{p}_s^{\alpha_s} r_0\ \ (\tilde{p}_j, r_0)=1\\
&q=\tilde{p}_1^{\beta_1}\cdots\: \tilde{p}_s^{\beta_s} q_0\ \ (\tilde{p}_j, q_0)=1\\
&h=\tilde{p}_1^{\gamma_1}\cdots\: \tilde{p}_s^{\gamma_s}\:,\ \text{with prime numbers $\tilde{p}_i$,}\tag{3.11}
\end{align*}
where $\alpha_j$, $\beta_j$ and $\gamma_j$ are positive integers with $\gamma_j=\min(\alpha_j, \beta_j)$, $j=1,\ldots, s$. Define
\[
h_1=\tilde{p}_1^{\delta_1}\cdots\: \tilde{p}_s^{\delta_s}\:,\tag{3.12}
\]
where $\delta_j=\alpha_j$ or $0$ according as $\alpha_j=\gamma_j$ or not. Then
$h_1\mid h$. Write
\[
h_2=\frac{h}{h_1}\:.\tag{3.13}
\]
Then
\[
h_1h_2=h,\ (h_1, h_2)=1,\ \left(\frac{r}{h_1},\:\frac{q}{h_2}\right)=1\:.\tag{3.14}
\]
\end{definition}
\begin{lemma}\label{lem32}
Let $d\mid k$, $g\mid k$ and $(m,k)=(f,k)=1$. Let also $\chi\bmod\: g$ be induced by the primitive character $\chi^* \bmod g^*$. Then 
$$|G(d, f, m, \chi_g, k)|\leq {g^*}^{1/2}\:.$$
\end{lemma}
\begin{proof}
This is Lemma 3 of \cite{liu}.
\end{proof}
\begin{lemma}\label{lem33}
Let $d\mid k$ and $(m,k)=(f,k)=1$. Let also $\chi^0(\bmod\: k)$ be the principal character. Then for $(d, k/d)>1$,
$$G(d, f, m, \chi^0)=0$$
and for $(d, k/d)=1$,
$$G(d, f, m, \chi^0)=\mu\left(\frac{k}{d}\right)e\left( \frac{fmt}{d}\right)\:,$$
where $t$ is defined by $tk/d\equiv 1(\bmod\: d)$.
\end{lemma}
\begin{proof}
This is Lemma 7 of \cite{liu}.
\end{proof}
The basic identity for the asymptotic evaluation of the major arcs is the following.
\begin{lemma}\label{lem34}
Let $a, q, r$ be positive integers, and $h, h_1, h_2$ defined as in (3.10)--(3.13), such
that (3.14) holds. Then 
\begin{align*}
S\left(\frac{a}{q}+\lambda, r, b\right)&=\frac{1}{\phi(r/h_1)\phi(q/h_2)}\sum_{\xi(\bmod\: r/h_1)}\overline{\xi}(b)\sum_{\eta(\bmod\: q/h_2)}G(h, b, a, \overline{\eta}, q)\\
\ \ &\times  \sum_{n\leq N}\xi\eta(n)\Lambda(n)e(n\lambda)+O(L^2)\:,
\end{align*}
where 
$$\sum_{\xi(\bmod r/h_1)}$$ 
respectively 
$$\sum_{\eta(\bmod q/h_2)}$$ 
are
over the Dirichlet characters $\bmod\:r/h_1$ respectively $q/h_2$.
\end{lemma}
\begin{proof}
This is Lemma 2 of \cite{liu}.
\end{proof}
We now decompose the sum in Lemma \ref{lem34} in three partial sums.
\begin{definition}\label{def33}
$$S\left(\frac{a}{q}+\lambda r, b\right)=S_0(a, q, \lambda, r, b)+S_1(a, q, \lambda, r, b)+S_2(a, q, \lambda, r, b)\:,$$
where $S_0$, $S_1$, respectively $S_2$ are the sums corresponding to\\
(0) $\xi=\xi^0\left(\bmod \frac{r}{h_1}\right),\ \eta=\eta^0\left(\bmod\: \frac{q}{h_2}\right)$\\
(i) $\xi=\xi^0\left(\bmod \frac{r}{h_1}\right),\ \eta\neq \eta^0\left(\bmod\: \frac{q}{h_2}\right)$\\
(ii) $\xi\neq \xi^0\left(\bmod \frac{r}{h_1}\right)$,\\
respectively
($\xi^0$ resp. $\eta^0$ are the principal characters $\bmod\: r/h_1$ resp. $\bmod\: q/h_2$).
\end{definition}
The asymptotics of the major arcs contribution and thus also of (3.6)  comes from the
sum $S_0$. To state the result we need the following definitions:
\begin{definition}
Let $q, r$ be positive integers and $(q, r)=h$. For $(a, q)=1$ and $(b,r)=1$ define
$$f(r, q, a, b):=\left\{ 
  \begin{array}{l l}
    \frac{\mu(q/h)}{\phi(rq/h)}\: e\left(\frac{abt}{h}\right)\:, & \quad \text{if $\ (q/h, h)=1$, $tq/h\equiv 1(\bmod\: h)$}\vspace{2mm}\\ 
      0\:, & \quad \text{if $\ (q/h, h)>1$.}
  \end{array} \right.$$
\end{definition}
Then we have:
\begin{lemma}\label{lem35}
$$S_0=f(r, q, a, b)\sum_{n\leq N}e(n\lambda)+O(|\lambda| N^{3/2}\log^3 N)+O(N^{1/2}\log^3 N)\:.$$
\end{lemma}
\begin{proof}
\begin{align*}
S_0&=\frac{1}{\phi(r/h_1)\phi(q/h_2)}\: G(h, b, a, \overline{\eta}, q)\:\sum_{n\leq N}\chi^0(n)\Lambda(n)e(n\lambda)\\
&=\frac{1}{\phi(rq/h)}\:G(h, b, a, n_0, q/h_2)\sum_{n\leq N}e(n\lambda)+\sum_{n\leq N}\chi^0(n)(\Lambda(n)-1)e(n\lambda)\\
&\ \ +O\left(\frac{L^2}{\phi(rq/h)}\right)\:.
\end{align*}
By the Riemann Hypothesis and partial summation we have
$$\sum_{n\leq N}\chi^0(n)(\Lambda(n)-1)e(n\lambda)=O(|\lambda|N^{3/2}\log^3 N)+O(N^{1/2}\log^3 N)\:.$$
The result now follows from Lemma \ref{lem33}.
\end{proof}
We now show that for $\frac{1}{2}R\leq r\leq 2R$ the contribution of $S_1$ and 
$S_2$ to the major arcs contribution is negligible.
\begin{definition}\label{def35}
For $(b_1, b_2, b_3)\in\mathbb{Z}^3$ let
$$\int_0:=\int_{E_1(R)}S(\alpha, r, b_1)S(\alpha, r, b_2)S(\alpha, r, b_3)e(-N\alpha)d\alpha$$
and
$$\int_1:=\int_{E_1(R)}S_0(\alpha, r, b_1)S_0(\alpha, r, b_2)S_0(\alpha, r, b_3)e(-N\alpha)d\alpha\:.$$
\end{definition}
\begin{lemma}\label{lem36}
Let $\frac{1}{2}R\leq r\leq 2R$, $(b_1, b_2, b_3)\in\mathbb{Z}^3$. Then 
$$\int_0=\int_1+\:O(N^{11/4}R^{-2})\:.$$
\end{lemma}
\begin{proof}
We partition the sums $S(\frac{a}{q}+\lambda, r, b_i)$ according to Definition
\ref{def33}. The product 
$$\prod_{i=1}^3 S(\alpha, r, b_i)$$ 
becomes a sum of 
27 products 
$$S_{j_1}(\alpha, r, b_1)S_{j_2}(\alpha, r, b_2)S_{j_3}(\alpha, r, b_3)$$ 
with 
$j_i\in\{0, 1, 2 \}$.\\
Assume that $(j_1, j_2, j_3)\neq (0,0,0)\:.$ Without loss of generality we may assume
$j_1\neq 0$. Then the characters $\xi \eta$ appearing in the sums 
$$\sum_{n\leq N} \xi \eta(n)\Lambda(n)e(n\lambda)$$
by Lemma \ref{lem34}  are non-principal. Therefore by the GRH we have the estimate
$$\sum_{n\leq N} \xi \eta(n)\Lambda(n)e(n\lambda)=O(N^{1/2}\log^3(NR))+O(|\lambda|N^{\:3/2}\log^3(NR))\:.$$
The other sums may be trivially estimated by
$$S_{j_i}=O(NR^{-1})\:.$$
This proves Lemma \ref{lem36}.
\end{proof}
\begin{lemma}\label{lem37}
$$\int_0=\left(\frac{1}{2}N^2+O(N^2L^{-C})\right)\sum_{q\leq P}\sum_{\substack{a=1 \\ (a,q)=1}}^q f(r, q, a, b_1)f(r, q, a, b_2)f(r, q, a, b_3)e\left(-\frac{aN}{q}\right)\:.$$
\end{lemma}
\begin{proof}
We obtain by Lemmas \ref{lem35} and \ref{lem36}:
\begin{align*}
\int_0&=\sum_{q\leq P}\sum_{a=1}^q f(r, q, a, b_1)f(r, q, a, b_2)f(r, q, a, b_3)
)e\left(-\frac{aN}{q}\right)\\
&\ \ \times \int_{|\lambda|\leq 1/qQ}\left( \sum_{n\leq N} e(n\lambda) \right)^3 e(-N\lambda)d\lambda+O(N^{11/4}R^{-2})\:.
\end{align*}
Using the estimate 
$$\sum_{n\leq N}e(n\lambda)\ll \min\left(N,\frac{1}{|\lambda|}\right)$$
one sees that the integral is 
\begin{align*}
\int_{-1/2}^{1/2}\left(\sum_{n\leq N}e(n\lambda)\right)^3e(-N\lambda)d\lambda+
O\left(\int_{1/qQ}^{1/2}\lambda^{-3}d\lambda\right)
&=\sum_{\substack{n_1+n_2+n_3=N\\ 1\leq n_j\leq N}}1+O((qQ)^2)\\
&=\frac{1}{2}N^2+O(N^2L^{-C})\:.
\end{align*}
We thus have proved Lemma \ref{lem37}.
\end{proof}
\begin{lemma}\label{lem38}
Let $b_1+b_2+b_3\equiv 0 (\bmod\: r)$. Then the singular series
$$\sigma(N;r):=\sum_{q=1}^\infty\sum_{a=1}^q f(r, q, a, b_1)f(r, q, a, b_2)f(r, q, a, b_3)
)e\left(-\frac{aN}{q}\right)$$
converges and has the value
$$\frac{1}{\phi^3(r)}\sum_{q=1}^\infty\frac{\mu(q/h)}{\phi^3(q/h)}\sum_{\substack{a=1\\ (a,q)=1}}^qe\left(\frac{a(b_1+b_2+b_3)t}{h}-\frac{aN}{q}  \right)\:.$$
We also have
\begin{align*}
\sigma(N;r)&=\frac{C(r)}{r^2}\prod_{p\mid r}\frac{p^3}{(p-1)^3+1}\prod_{\substack{p\mid N \\ p\nmid r}}\frac{(p-1)((p-1)^2+1)}{(p-1)^3+1} \prod_{p>2}\left(1+\frac{1}{(p-1)^3}\right)\\
&\geq \frac{1}{r^2\log r}\:.
\end{align*}
\end{lemma}
\begin{proof}
This is due to Rademacher \cite{Rademacher}
\end{proof}
\begin{lemma}\label{lem39}
Let $\epsilon>0$. Then we have
$$\frac{1}{\phi^3(r)}\sum_{\substack{q>P\\ (q/h, h)=1}}\frac{\mu(q/h)}{\phi^3(q/h)}\sum_{\substack{a=1 \\ (a, q)=1}}^q e\left(\frac{a(b_1+b_2+b_3)t}{h}-\frac{aN}{q}\right)
=O\left(\frac{r^{2+\epsilon}}{\phi^3(r)}P^{-1}\right)\:.$$
\end{lemma}
\begin{proof}
We use the estimates 
$$\sum_{\substack{a=1 \\ (a, q)=1}}^q e\left(\frac{a(b_1+b_2+b_3)t}{h}-\frac{aN}{q}\right)\leq q\:,$$
$$d(m)\leq m^\epsilon\ \text{for the divisor function}$$ 
and
$$\frac{m}{\phi(m)}=O(\log\log m)\:.$$
We obtain
\begin{align*}
&\frac{1}{\phi^3(r)}\sum_{\substack{q>P\\ (q/h, h)=1}}\frac{\mu(q/h)}{\phi^3(q/h)}\sum_{\substack{a=1 \\ (a, q)=1}}^q e\left(\frac{a(b_1+b_2+b_3)t}{h}-\frac{aN}{q}\right)\\
&\leq \frac{1}{\phi^3(r)} \sum_{\substack{q>P\\ (q/h, h)=1}}\frac{1}{\phi^3(q/h)}\cdot q \\
&\leq \frac{1}{\phi^3(r)} \sum_{h\mid r} \phi^3(h)\left(\sum_{\substack{q>P\\ h\mid q}}\frac{1}{q^2}\right)\log\log q\\
&\leq \frac{1}{\phi^3(r)} \sum_{h\mid r} \phi^3(h)\frac{1}{h^2}\sum_{\tilde{q}>P/h}\frac{1}{\tilde{q}^{\:2}}\\
&\leq \frac{1}{\phi^3(r)}  P^{-1}\sum_{h\mid r} h^2=O\left( \frac{r^{2+\epsilon}}{\phi^3(r)}\: P^{-1} \right)\:.
\end{align*}
\end{proof}
\begin{theorem}\label{thm3}
Let 
$$R\leq N^{1/20},\ b_1+b_2+b_3=0,\ \frac{R}{2}\leq R\leq 2R\:.$$ 
Then 
$$\sum_{\substack{(n_1, n_2, n_3)\in\mathbb{N}^3\\ n_i\equiv b_i(\mod\: r)\\ n_1+n_2+n_3=N}}\Lambda(n_1)\Lambda(n_2)\Lambda(n_3)=\frac{1}{2}\sigma(N;r)N^2(1+o(1))$$
as $N\rightarrow \infty$.
\end{theorem}
\begin{proof}
This follows from (3.6), Lemma \ref{lem31} and Lemmas \ref{lem37}, \ref{lem38},
\ref{lem39}.
\end{proof}
\section{Conclusion}
Theorem \ref{main} now follows if we apply Theorem \ref{thm3} with 
$$b_1=b_2=u_0,\ b_3=N-2u_0,\ r=P^*\:.$$

\end{document}